\documentclass[a4paper, 11pt, leqno]{amsart}
\usepackage{graphicx}
\usepackage[applemac]{inputenc}
\usepackage{selinput}
\usepackage{color}
\usepackage{hyperref}
\usepackage[spanish, english]{babel}
\usepackage{amssymb,tikz}
\usepackage{mathabx}
\usepackage{mathrsfs}
\usepackage{amsfonts,amscd,amsmath}
\usepackage{latexsym}
\usepackage{enumitem} 

\numberwithin{equation}{section}

\def\irr#1{{\Irr}(#1)}
\def\irra#1#2{{\Irr}_{#1}(#2)}

\def\irra#1#2{{\Irr}_{#1}(#2)}

\def\zent#1{{\bf Z}(#1)}
\def\syl#1#2{{\rm Syl}_{#1}(#2)}

\def\nor{\triangleleft}

\def\norm#1#2{{\bf N}_{#1}(#2)}
\def\cent#1#2{{\bf C}_{#1}(#2)}

\let\phi=\varphi
\def\sbs{\subseteq}

\newtheorem{lem}[subsection]{Lemma}
\newtheorem{cor}[subsection]{Corollary}
\newtheorem{teo}[subsection]{Theorem}

\newtheorem*{thm*}{Theorem}

\newtheorem*{thmA}{Theorem A}
\newtheorem*{thmB}{Theorem B}

\theoremstyle{definition}
\newtheorem{rem}[subsection]{Remark}
\theoremstyle{definition}

\theoremstyle{definition}

\theoremstyle{definition}

\newcommand{\Irr}{\operatorname{Irr}}

\begin{document}

\vspace*{-0.2cm}

\author{Adele Maltempo}
\address[A. Maltempo]{School of Mathematics, Watson Building, University of Birmingham, Edgbaston,
Birmingham B15 2TT, UK}
\email{AXM2451@student.bham.ac.uk}

\author{Carolina Vallejo}
\address[C. Vallejo]{Dipartamento di Matematica e Informatica `Ulisse Dini', Universit\`a di Firenze,
  50134 Firenze, Italy}
\email{carolina.vallejorodriguez@unifi.it}

\title[The McKay conjecture with coprime group automorphisms after Okuyama-Wajima]{The McKay conjecture with coprime group automorphisms and the Okuyama-Wajima argument}
\begin{abstract}
Let $A$ and $G$ be finite groups. Suppose that $A$ acts coprimely on $G$ stabilizing $N\nor G$. Let $\theta \in \irr N$ be $A$-invariant. We prove that the number of $A$-invariant irreducible characters of $G$ that lie over $\theta$ can be counted in terms of the $(A, \theta)$-good conjugacy classes of $G_\theta/N$, where $G_\theta$ is the inertia subgroup of $\theta$ in $G$. This result generalizes a classic result of Gallagher and can be used to prove the following: if $P$ is an $A$-invariant Sylow $p$-subgroup of $G$ and $G$ is $p$-solvable, then there exists an $A$-equivariant (McKay) bijection between the irreducible characters of degree prime to $p$ of $G$ and those of $\norm G P $. While this is a consequence of a recent result of D. Rossi,
our approach here is independent of Rossi's and follows the original idea of the proof of the McKay conjecture for $p$-solvable groups. In particular, we rely on the so-called Okuyama-Wajima argument to deal with characters above Glauberman correspondents.

 \end{abstract}
      \thanks{We thank Luis Pablo Colmenar for pointing out an error in the published version of this note, which is undergoing the publisher's retraction process. We are indebted to Noelia Rizo for useful conversations and a thorough reading of this revised manuscript}
\maketitle

\section*{Introduction}
\noindent

\noindent

Let $G$ be a finite group. When $N \nor G$ and $\theta\in \irr N$ is $G$-invariant, we say that an element $Ng \in G/N$ is $\theta$-good in $G$  if $\theta$ extends to $N \langle g , x \rangle$ for every $x \in G$ such that $N \langle g , x \rangle/N$ is abelian.
Notice that if $Ng \in G/N$ is a $\theta$-good, then every conjugate of $Ng$ is also $\theta$-good. We can therefore talk about $\theta$-good conjugacy classes of $G/N$.
A classical result by Gallagher \cite{Gal70} asserts that the number $|\irr{G|\theta}|$ of irreducible characters of $G$ lying over $\theta$ equals the number of $\theta$-good conjugacy classes of $G/N$. 
Gallagher's counting argument is very useful in character theory and it was used by Okuyama and Wajima to offer the first proof of the Alperin-McKay conjecture for $p$-solvable groups \cite{OW80}.

Different versions of Gallagher's counting argument have appeared in the literature. For instance, fixing a prime $p$, the number of irreducible $p$-Brauer character lying under a fixed $G$-invariant Brauer character of $N$ equals the number of $p$-regular $\theta$-good classes of $G/N$. For a set of primes $\pi$, if $G$ is $\pi$-separable and $N$ is a $\pi'$-group, the number $|{\rm I}_\pi(G|\theta)|$ of Isaacs $\pi$-partial characters over $\theta$ equals the number of $\theta$-good classes of $\pi$-elements of $G/N$ \cite[Theorem 6.2]{Isa88}. These results can be used to obtained McKay-like bijections through the Okuyama-Wajima argument as in \cite{Wol90}.  Not every generalization of Gallagher's counting argument can be used to prove new properties of McKay bijections for $p$-solvable groups. Recently, a different generalization of Gallagher's counting theorem has been established by J. Murray. In \cite[Thm. A]{Mur23}, it is shown that these ideas can be extended to count the number of real characters of a group $G$
lying over a fixed character of a normal subgroup. However, 
the group ${\rm GL}(2,3)$ for $p=3$ shows that, in general, it is not possible to construct McKay bijections that preserve real characters.

Suppose that $A$ acts coprimely on $G$. Let $N \nor G$ be $A$-invariant. We show that we can count the number $|\irra A {G|\theta}|$ of $A$-invariant irreducible characters of $G$ lying over some $A$-invariant $\theta\in \irr N$ in terms of certain conjugacy classes of $G/N$
 that we will call $(A, \theta)$-good. We use this result to prove that whenever $A$ acts coprimely on a $p$-solvable group $G$, there exist $A$-equivariant McKay bijections for any prime $p$.
This is not a new result, as it follows from a recent theorem of Rossi \cite[Thm. B]{Ros23}. 
Our approach, however, is different and relies on more elementary character-theoretic arguments.
 Indeed, the proof of \cite[Thm. B]{Ros23} for $p$-solvable groups depends on results from \cite[\S 5]{NS14},
  which in turn follow from the work of F. Ladisch \cite[Corollary 11.3]{Lad11}. 
  Of course, using these more sophisticated arguments and \cite{CS25}, one can derive the existence of $A$-equivariant McKay bijections for $G$ whenever $A$ acts on $G$ by automorphisms and stabilizes a Sylow $p$-subgroup of $G$ (that is, without the coprimality hypothesis we impose here).
Nonetheless, for the case where $G$ is $p$-solvable and the action of $A$ is coprime, our method yields a useful simplification.

Under the above assumptions, if $\theta \in \irr N$ is $G$-invariant, we say that an element $Ng \in G/N$ is $(A, \theta)$-good in $G/N$
if $Ng$ is $A$-invariant (that is if $Ng^a=Ng$ for every $a \in A$) and for every $A$-invariant $Nx \in G/N$ such that $N\langle g, x \rangle /N$ is abelian, $\theta$ extends to $N\langle g, x \rangle$.
Notice that if the action of $A$ on $G$ is trivial, then the $(A, \theta)$-good elements are precisely the $\theta$-good elements.
We will say that a conjugacy class of $G/N$ is $(A, \theta)$-good if it contains $(A, \theta)$-good elements.
Notice that if $Ng$ is an $(A, \theta)$-good element, it is not true that every $Ng^x$ for $x \in G$ is $(A, \theta)$-good. This contrast the $\theta$-good setting.
On the other hand, $(A, \theta)$-good conjugacy classes are $A$-invariant. Our first result is the following.

\begin{thmA} Let $A$ act coprimely on $G$. Let $N \nor G$ be $A$-invariant. Let $\theta \in \irr N$ be $A$-invariant.
Then the number $|\irra A{G|\theta}|$ of irreducible $A$-invariant characters of $G$ lying over $\theta$ equals the number of $(A, \theta)$-good conjugacy classes of $G_\theta/N$, where $G_\theta$ is the stabilizer of $\theta$ in $G$.
\end{thmA}

Recall that, for a prime $p$ and a finite group $G$, we denote by $\irra {p'} G$ the set of irreducible characters of $G$ whose degree 
is coprime to $p$.
As mentioned above, thanks to the so-called Okuyama-Wajima argument \ref{thm:OW}, we can use Theorem A to prove the following.

\begin{thmB} Let $A$ act coprimely on $G$. Let $p$ be a prime  and $P \in \syl p G$ be an $A$-invariant Sylow $p$-subgroup of $G$. 
Suppose that $G$ is $p$-solvable.
There exists an $A$-equivariant bijection
$$\irra {p'} {G} \to \irra {p'} {\norm G P}\, .$$
\end{thmB}

This paper is structured as follows. In Section \ref{Gallagher} we prove Theorem A. In Section \ref{OW}, we use Theorem A to obtain some consequences of the Okuyama--Wajima argument, specifically Theorem \ref{thm:OWrevisited}. Finally, in Section \ref{reduction}, we prove Theorem B. We follow the notation \cite{Isa76, Nav18}, with \cite{Nav18} serving as a guiding reference for our proofs.

\section{Counting $A$-invariant characters after Gallagher, Glauberman and Isaacs}\label{Gallagher} 

\noindent Our aim in this section is to prove Theorem A from the Introduction. Let $N\nor G$ and $\theta \in \irr N$. We write
$\irr{G|\theta}$ to denote the subset of $\irr G$ consisting of characters $\chi$ such that $[\chi_N, \theta]\neq 0$, and we often refer to characters in $\irr {G|\theta}$ as characters lying over $\theta$. Suppose that $\theta$ is $G$-invariant,
Gallagher's counting argument asserts that the number $|\irr{G|\theta}|$ equals the number of $G/N$ conjugacy classes consisting of $\theta$-good elements in $G/N$, where an element $Ng \in G/N$ is said to be $\theta$-good in $G/N$ if all the extensions of $\theta$ to $N\langle g \rangle$ are $D$-invariant for $D/N=\cent {G/N}{Ng}$. This is equivalent to ask that $\theta$ extends to $N \langle g , x \rangle$ for every $x \in G$ such that $N \langle g , x \rangle/N$ is abelian; by \cite[Lem. 5.13]{Nav18}.
The original proof of Gallagher follows by adapting a computation of I. Schur and can be found in \cite[\S 2]{Gal70}. 

\smallskip

%
%

Suppose that $A$ acts coprimely on $G$ stabilizing $N \nor G$.
Let $\theta\in \irr N$ be $GA$-invariant, where $GA$ denotes the semidirect product of $G$ and $A$.
Recall that we say that an element $Ng \in G/N$ is $(A, \theta)$-good in $G/N$, if $Ng$ is $A$-invariant and for every $A$-invariant $Nx \in G/N$ 
such that $N\langle g, x \rangle /N$ is abelian, $\theta$ extends to $N\langle g, x \rangle$.
Moreover we say that a conjugacy class of $G/N$ is $(A, \theta)$-good if it contains some $(A, \theta)$-good element.
As already mentioned, even when $Ng$ is $(A, \theta)$-good not every element in the conjugacy class of $Ng$ is $(A, \theta)$-good.
Notice also that $(A, \theta)$-good conjugacy classes contain $A$-invariant elements, and hence they are $A$-invariant subsets themselves. 


\smallskip

We restate and prove the statement of Theorem A here for future reference within this text. 
The proof depends on properties of the Glauberman--Isaacs correspondence as well as on Gallagher's counting argument.
Recall that when $A$ acts coprimely on $G$, there is a one-to-one correspondence between the set $\irra A G$ and the set $\irr{\cent G A}$, 
where $\cent G A=\{ g \in G \ | \ g^a =g \text{ for all } a \in A \}$. This correspondence is known as
 the Glauberman--Isaacs correspondence. 
 For properties of the Glauberman--Isaacs correspondence we refer the reader to \cite{Wol79}.
 We will denote by $GA$ the semidirect product of $G$ and $A$ with respect to the prescribed action.

\begin{teo}\label{thm:A-invariant} 
Let $A$ act coprimely on $G$. Let $N \nor G$ be $A$-invariant. Let $\theta \in \irr N$ be invariant under the action of $GA$.
Then the number $|\irra A{G|\theta}|$ of irreducible $A$-invariant characters of $G$ lying over $\theta$ equals the number of conjugacy classes of $G/N$ that contain $(A, \theta)$-good elements in $G/N$. 
\end{teo}

\begin{proof} Write $C=\cent G A$.
Since $A$ acts coprimely on $G$, we have that $\cent {G/N} A = NC/N$ by \cite[Corollary 3.28]{Isa08}. Let $\theta^* \in \irr{N \cap C}$ be the Glauberman--Isaacs correspondent of $\theta$. Applying \cite[Lemma 2.5]{Wol79} twice, we have that 
\[ |\irra A {G|\theta}| = |\irr{C|\theta^*}| = |\irra A {NC|\theta}|\, . \] 
Using \cite[Corollary 1.4]{Wol79} and the main theorem in \cite[\S 2]{Gal70}, we have that $|\irra A {G|\theta}| = |\irr {NC|\theta}|$ equals the number of $\theta$-good conjugacy classes of $NC/N$. By \cite[Theorem 3.26]{Isa08}, the map $\mathcal{K} \to \mathcal{K} \cap NC/N$ defines a bijection between the set of $A$-invariant conjugacy classes of $G/N$ and the set of conjugacy classes of $NC/N$. 

We claim that this bijection sends the set of conjugacy classes of $G/N$ that contain $(A, \theta)$-good elements onto the set of conjugacy classes of $NC/N$ 
consisting of $\theta$-good elements. 
Let $Ng$ be a $(A, \theta)$-good element in a conjugacy class $\mathcal K$ of $G/N$. Then $\mathcal K$ is $A$-invariant and
$Ng\in NC/N$ so $Ng=Nc$ for some $c \in C$.  In particular, $N\langle g \rangle=N\langle c \rangle$, and $\mathcal K \cap NC/N$ is the conjugacy class of $Nc$ in $NC/N$.
For every $Nx \in NC/N$ such that $N\langle c, x\rangle /N$ is abelian, we have that $N\langle c, x \rangle=N\langle c \rangle N \langle x \rangle=N\langle g, x \rangle$. By assumption $\theta$ extends to $N\langle c, x \rangle$, so $Nc$ is $\theta$-good in $NC/N$.
Now, if $Nc$ is $\theta$-good in $NC/N$ we have that the conjugacy class of $G/N$ containing $Nc$ is $A$-invariant. Let $Nx \in G/N$ be $A$-invariant such that $N\langle c, x \rangle/N$ is abelian. Since $Nx \in NC/N$ and $Nc$ is $\theta$-good in $NC/N$, we have that $\theta$ extends to $N\langle c, x \rangle$. Hence $Nc$ is $(A, \theta)$-good. 
\end{proof}


\begin{rem} Unfortunately, the of definition of $(A, \theta)$-good elements appears to be useful only when the action of $A$ on $G$ is coprime, although it could also work when only the action of $A$ on $G/N$ is assumed to be coprime.
Say $Ng\in G/N$ is $(A, \theta)$-good if $Ng\in G/N$ is $A$-invariant and for every $A$-invariant $Nx \in G/N$ such that $N\langle g, x \rangle/N$ is abelian
$\theta$ admits an $A$-invariant extension $\mu$ to $N\langle g, x \rangle$. 
The analogue of Theorem \ref{thm:A-invariant} does not generally hold with this definition.
For instance, consider the
alternating group $G={\sf A}_6$, $N=1$ and the action of $A=\langle (1 2)\rangle \leq {\sf S}_6 $ on $G$ by conjugation. 
Then $|\irra A G|$ equals the number of $A$-invariant conjugacy classes of $G$ by Brauer's lemma on the character table \cite[Theorem 2.3]{Nav18}. 
However, not every $A$-invariant conjugacy class of $G$ contains some $A$-invariant element.  The conjugacy class of $(1 2 3 )(4 5 6)$ intersects trivially $\cent {G} A=G_{\{1,2\}}\rtimes \langle (12) (34) \rangle$ (here $G_{\{1,2\}}$ is the subgroup that fixes both $1$ and $2$). 
\end{rem}

We can derive the following statement that, under certain circumstances, allows us to count $|\irra A {G|\theta}|$ more generally when $\theta$ is not assumed to be $GA$-invariant but we know that the set $\irra A {G|\theta}$ is non-empty. When $A$ acts on $G$ by automorphisms stabilizing $N \nor G$, we have that $N$ and $G$ are normal subgroups of $GA$. Moreover, whenever $\theta \in \irr N$ we have that $\irra A {G|\theta}=\irra{GA}{G|\theta}$. 
We adopt this perspective in the following results.

\begin{cor}\label{cor:A-invariant}
Let $N$ and $G$ be normal subgroups of a group $A$ with $N\sbs G$. Let $\theta \in \irr N$. Suppose that $(|A_\theta/G_\theta|, |G_\theta|)=1$. 
If $ \irra A {G|\theta}$ is non-empty,
then the number $|\irra A{G|\theta}|$ equals the number of  $(D, \theta)$-good conjugacy classes of $G_\theta/N$,
where $D$ is a complement of $G_\theta$ in $A_\theta$. 
\end{cor}
\begin{proof} We first show that under our hypotheses  $\displaystyle \irra{A}{G|\theta}=\irra{A_{\theta}}{G|\theta}$.

 Let $\chi \in \irra A {G|\theta}$. For every $a \in A$, we have that $\theta$ and $\theta^a$ lie under $\chi$. Hence there is some $g \in G$ such that $\theta^g =\theta^a$.
       In particular $ag^{-1}\in A_\theta$. We have shown that $A=GA_\theta$.
We can therefore conclude that $\irra{A}{G|\theta}=\irra{A_{\theta}}{G|\theta}$. 

By the Clifford correspondence, we know that character induction induces an $A_\theta$-equivariant bijection
    \begin{align*}
        \irr{G_{\theta}|\theta} &\rightarrow \irr{G|\theta}\, .
    \end{align*}
In particular, 
       $$ |\irra{A_{\theta}}{G_{\theta}|\theta}|=|\irra{A_{\theta}}{G|\theta}|\, .$$
Using the Schur-Zassenhaus theorem $A_\theta=G_\theta D$, where $D$ is a complement of $G_\theta$. 
We can apply Theorem \ref{thm:A-invariant} that tells us that 
\[ |\irra A {G|\theta}|=|\irra {A_\theta}{G_\theta|\theta}|=|\irra D{G_\theta |\theta}|\] 
is the number of $(D, \theta)$-good conjugacy classes of $G_\theta/N$.
\end{proof}

We care to mention that the equality $|\irra A{G|\theta}|=|\irra {A_\theta}{G_\theta|\theta}|$ does not always hold as shown by the group $A={\sf C}_2 $ acting on $G={\sf C}_3\times {\sf C}_3$ with $\cent G A=1$, $N$ a proper $A$-invariant subgroup of $G$, and $\theta \in \irr N$ faithful. In this case $|\irra A {G|\theta}|=0$ while $|\irra {A_\theta}{G_\theta | \theta}|=|\irr{G|\theta}|=3$. 

When the action of $A$ on $G$ is coprime, the condition that $\irra A{G|\theta}$ is non-empty  is satisfied, for instance, if $\theta$ is $A$-invariant by \cite[Theorem 13.28 and Corollary 13.30]{Isa76}. 

%
%
%
%
%
%
%

\section{The Okuyama--Wajima argument revisited}\label{OW}

\noindent The aim of this section is to apply Theorem \ref{thm:A-invariant} to derive a new consequence of the Okuyama-Wajima argument, namely Theorem \ref{thm:OWrevisited}.

The Okuyama-Wajima argument concerns the character theory above Glauberman correspondents. Recall that when a $p$-group $Q$ acts by automorphisms on a group $K$ of order coprime to $p$, then the Glauberman correspondence is a canonical bijection $^*\colon \irra Q K \to \irr {\cent K Q}$ between the set $\irra Q K$ of $Q$-invariant irreducible characters of $K$ and the set $\irr{\cent K Q}$ of irreducible characters of the fixed-point subgroup under the action. In fact, given $\theta \in \irra Q K$, its Glauberman correspondent $\theta^*$ is the unique irreducible constituent of the restriction $\theta |_{\cent K Q}$ of $\theta$ to $\cent K Q$ appearing with multiplicity coprime to $p$ (see, for instance, \cite[Thm. 2.9]{Nav18}). When the group acting is a $p$-group, then  the Glauberman--Isaacs correspondence is precisely the Glauberman correspondence. 

\begin{teo}[The Okuyama--Wajima Argument]\label{thm:OW}
  Let $G$ be a group, let $p$ be a prime. Suppose $K \nor G$ has order not divisible by $p$. Let $Q$ be a $p$-subgroup of $G$ such that $KQ \nor G$. Let $H=\norm{G}{Q}$ and $C=\cent{K}{Q}$. Let $\theta \in \irr{K}$ be $Q$-invariant and let $\theta^*\in \irr{C}$ be the Glauberman correspondent of $\theta$. Suppose that $C \leq U \leq H$ is such that $U/C$ is abelian. Then $\theta$ extends to KU if and only if $\theta^*$ extends to U.
\end{teo}
\begin{proof} See \cite[Thm. 2]{OW80}. Notice that in the statement of \cite[Thm. 2]{OW80}, $Q$ is a Sylow $p$-subgroup of $G$ and $U=H$, but these stronger hypotheses are not necessary (using \cite[Thm. 5.10 and Cor. 6.2]{Nav18}, for instance). A more detailed proof can be found in \cite[Thm. 8.6]{Nav18}. 
\end{proof}

The statement of Theorem~\ref{thm:OW} holds more generally under the weaker assumption that $K$ and $Q$ have coprime orders, by work of Wolf~\cite[Thm.~1.5]{Wol90} (for the case where $Q$ is solvable see also \cite[Thm. 6.10]{Isa18}). 
It also remains valid when the quotient 
$U/C$ is not assumed to be abelian; however, the proof of this seemingly innocuous generalization requires completely different ideas and a fairly great amount of work, see for instance \cite[Cor. 11.3 and Thm. 4.3]{Lad11} or \cite[Thm. 6.5]{Tur08} combined with \cite[Thm. 7.12]{Tur09}. When the prime $p=2$, T. Wolf found a character-theoretical proof of this fact \cite[Thm. 1.5]{Wol90}, relying on properties of the Isaacs correspondence \cite{Isa73}. 

As mentioned in the Introduction and for the purpose of this note, the Okuyama-Wajima argument will suffice. We need to derive some consequences from it using the results contained in the previous section. Before doing so, we need to prove an easy, yet useful, lemma. Given groups $K, G\nor A$ with $K\sbs G$, $P\in \syl p G$, and a $P$-invariant $\theta \in \irr K$, we will be interested in counting the number
$$|\irra{p',A}{G|\theta}|$$
of irreducible $A$-invariant characters of $G$ of degree coprime to $p$ that lie over $\theta$. 
The following technical lemma, along with several ideas from its proof, will be used throughout the remainder of this section.

\begin{lem}\label{lem:technical}
    Let $A$ be a group, and let $G \nor A$. Suppose that $B \leq A$ is such that $A=GB$, and write $H=G \cap B$. Let $K \sbs G$, with $K \nor A$.
    Assume that $G=KH$ and write $C=K\cap H$.
     Let $\theta \in \irr{K}$ and let $\phi \in \irr{C}$. 
     If $B_{\theta}=B_{\phi}$
and
    \begin{equation*}
        |\irra{A_{\theta}}{G_{\theta}|\theta}|=|\irra{B_{\theta}}{H_{\phi}|\phi}|,
    \end{equation*}
    then we have that 
    \begin{equation*}
        |\irra{A}{G|\theta}|=|\irra{B}{H|\phi}|.
    \end{equation*}
Let $p$ be a prime and $P \in \syl p {G}$. Assume that $Q=P\cap H \in \syl p H$ and that $\theta$ is $P$-invariant. 
If $B_{\theta}=B_{\phi}$ and
    \begin{equation*}
        |\irra{p',A_{\theta}}{G_{\theta}|\theta}|=|\irra{p',B_{\theta}}{H_{\phi}|\phi}|,
    \end{equation*}
    then we have that
    \begin{equation*}
        |\irra{p',A}{G|\theta}|=|\irra{p',B}{H|\phi}|.
    \end{equation*}
\end{lem}
\begin{proof}
%
     We prove the first statement. Assume that $\irra{B}{H|\phi}$ is not empty. Arguing as in Corollary \ref{cor:A-invariant}, we have that $B=HB_\phi=HB_\theta$. Hence $A=GB=G(H B_\theta)=GA_\theta$. 
  Then  $|\irra A {G|\theta}|=|\irra {A_\theta}{G|\theta}|=|\irra{A_\theta}{G_\theta|\theta}|$ and $|\irra {B_\phi}{H_\phi|\phi}|=|\irra{B_\phi}{H|\phi}|=|\irra{B}{H|\phi}|$, and we are done. 
  Assume now that $\irra{A}{G|\theta}$ is non-empty.
  Arguing as in Corollary \ref{cor:A-invariant}, we have that $A=GA_\theta=KHA_\theta=HA_\theta$ and $|\irra A {G|\theta}|=|\irra {A_\theta}{G|\theta}|=|\irra{A_\theta}{G_\theta |\theta}|$. 
  Then, by Dedekind's lemma, $B=B\cap (HA_\theta)=H(B\cap A_\theta)=HB_\theta=HB_\phi$ and therefore $|\irra {B}{H|\phi}|=|\irra {B_\phi}{H|\phi}|=|\irra{B_\phi}{H_\phi|\phi}|$, and again we are done.
  If we are not in any of the previous situations then both sets are empty and the first statement trivially holds.

In order to prove the second statement we follow the same ideas. Notice that under the hypotheses, $\varphi$ is $Q$-invariant and character induction actually yields bijections
    \begin{align*}
        \irra {p', A_\theta} {G_{\theta}|\theta} &\rightarrow \irra {p', A_\theta}{G|\theta}, \\
        \irra{p', B_\phi}{H_{\phi}|\phi} &\rightarrow \irra{p', B_\phi}{H|\phi},
    \end{align*}
    since $|G:G_\theta|$ and $|H:H_\phi|$ are coprime to $p$.
\end{proof}

\begin{rem}\label{rem:technical} Under the hypotheses of Lemma \ref{lem:technical} notice that we have shown that if either $\irra A {G|\theta}$ and $\irra B {H|\varphi}$ are non-empty
then we can write $B=HB_\varphi$ and $A=GA_\theta$.
\end{rem}

The following is a consequence of the Okuyama--Wajima argument and Theorem \ref{thm:A-invariant}. 

 \begin{cor}\label{cor:OW} 
Suppose that $D$ acts coprimely on $G$. Write $A = GD$ to denote the semidirect product of $G$ and $D$. 
Let $p$ be a prime. Let $K \nor A$ be a subgroup of order not divisible by $p$ contained in $G$. Let $P$ be a $D$-invariant $p$-subgroup of $G$. Suppose that $KP \nor G$. Write $C = \cent K P$. Let $\theta \in \irr{K}$ be $P$-invariant and let $\theta^* \in \irr{C}$ be the Glauberman correspondent of $\theta$. If $C \leq S \leq \norm G P$ is $D$-invariant, then 
\[ |\irra D {KS|\theta}| = |\irra D{S|\theta^*}|\, . \]
\end{cor}

\begin{proof} Since $P\in \syl p {KP}$,
   by the Frattini argument $G=K \norm G P$.  Notice that because $P$ is $D$-invariant, $KP\nor A$.
We have that $K\nor KSD$ and $K \nor KS$, so $C \nor S$ and $C \nor SD$.
  The interesting part of our set-up can be visualized in the following diagram.     
   
   \begin{center}
   \begin{tikzpicture}
  
    \node (left1) at (0,0) {$KSD$};
    
    \node [below left of=left1, yshift= -0.2cm, xshift=-0.5cm] (left2) {$KS$};
    \node [below left of=left2, yshift= -0.2cm, xshift=-0.5cm] (left4) {$\theta \ K$};

    \node [below right of=left1, yshift= -0.2cm, xshift=0.5cm] (right2) {$SD$};
    \node [below left of=right2, yshift= -0.2cm, xshift=-0.5cm] (right3) {$S$};
    \node [below left of=right3, yshift= -0.2cm, xshift=-0.5cm] (left5) {$\theta^* \ C$};

    \draw [black, line width=0.5pt, shorten <=-2pt, shorten >=-2pt] (left1) -- (right2);
    \draw [black, line width=0.5pt, shorten <=-2pt, shorten >=-2pt] (left1) -- (left2);
    \draw [black, line width=0.5pt, shorten <=-2pt, shorten >=-2pt] (left2) -- (right3);
    \draw [black, line width=0.5pt, shorten <=-2pt, shorten >=-2pt] (right2) -- (right3);
    \draw [black, line width=0.5pt, shorten <=-2pt, shorten >=-2pt] (left2) -- (left4);
    \draw [black, line width=0.5pt, shorten <=-2pt, shorten >=-2pt] (left5) -- (left4);
    \draw [black, line width=0.5pt, shorten <=-2pt, shorten >=-2pt] (left5) -- (right3);
\end{tikzpicture}

\end{center}
Notice that $\irra D {KS|\theta}=\irra{KSD}{KS|\theta}$ and $\irra{D}{S|\theta^*} = \irra {SD}{S|\theta^*}$. 

We want to apply Lemma \ref{lem:technical} with $KSD$ playing the role of $A$, $KS$ playing the role of $G$, $SD$ playing the role of $B$ and $S$ playing the role of $H$.
    Note that $(SD)_{\theta}=(SD)_{\theta^*}$ by \cite[Lem. 2.10]{Nav18}.  In particular, $S_\theta=S_{\theta^*}$.
   We also have that $S \nor SD$, $K \nor KSD$, $KSD=(KS)D$, and $S=KS \cap SD$. 
  If the sets $\irra{KSD}{KS|\theta}$ and $ \irra {SD}{S|\theta^*}$ are both empty, then 
  there is nothing to prove. We can therefore assume one of the two is non-empty.
  By Remark \ref{rem:technical}, we then have that $KSD=KS(KSD)_\theta$ and $SD=S(SD)_\theta$.
   By Lemma \ref{lem:technical}, recall that $S_\theta=S_{\theta^*}$, it will be enough to prove that
   \[ |\irra {(KSD)_\theta} {(KS)_\theta|\theta}|=| \irra{(SD)_\theta}{S_\theta|\theta^*}| \, . \]
 Since $(|(SD)_\theta:S_\theta|, |S_\theta|)=1$, 
by the Schur-Zassenhaus theorem we can choose a complement $E$ of $S_\theta$ in $(SD)_\theta$. 
Then $(SD)_\theta=S_\theta E$ and also $(KSD)_\theta=KS_\theta E$. 
It will be enough to prove that
   \[ |\irra {E} {(KS)_\theta|\theta}|=| \irra{E}{S_\theta|\theta^*}| \, . \]
 Since $E\sbs SD\sbs \norm A P$, we have that $P$ is $E$-invariant. Notice that $E$ acts coprimely on $G_\theta$ and $PK\nor G_\theta E$.
We can therefore assume both $\theta$ and $\theta^*$ to be $SD$-invariant.
  
We have that $(KSD,K,\theta)$ is a character triple with $K \leq KS \nor KSD$, and $(SD,C,\theta^*)$ is a character triple with $C \leq S \nor SD$. 
By Theorem \ref{thm:A-invariant}, we have that $|\irra{D}{KS|\theta}|$ is the number of $(D, \theta)$-good conjugacy classes of $KS/K$,
 while $|\irra{D}{S|\theta^*}|$ is the number of $(D, \theta^*)$-good conjugacy classes of $S/C$. Notice that the map
    \begin{align*}
        f\colon S/C &\rightarrow KS/K, \\
        Cs &\mapsto Ks
    \end{align*}
    is a $D$-equivariant isomorphism.
    We need to prove that the conjugacy class of $Cs$ in $S/C$ contains 
    $(D, \theta^*)$-good elements if, and only if, the conjugacy class of $Ks$ in $KS/K$ contains $(D, \theta)$-good elements.
Notice that $f$ maps $D$-invariant elements onto $D$-invariant elements.
 
    Fix $s \in S$ such that $Cs$ is $D$-invariant. We prove that $\theta^*$ extends to  $\langle C,s,x \rangle$, for every $x \in S$ such that $Sx$ is $D$-invariant and $\langle C,s,x \rangle/C$ is abelian if, and only if, $\theta$ extends to $\langle K,s,x \rangle$, for every $x \in S$ such that $Ks$ is $D$-invariant and $\langle K,s,x \rangle/K$ is abelian. Let $x \in S$ be such that $Cx$ is $D$-invariant and $\langle C,s,x \rangle/C$ is abelian and write $U=\langle C,s,x \rangle$, then $C \leq U \leq S$. We need to prove that $\theta$ extends to $KU$ if, and only if, $\theta^*$ extends to $U$. This is true by Theorem \ref{thm:OW}.
    \end{proof}

Finally, we can generalize \cite[Thm. 8.11]{Nav18} taking into account the action of coprime group automorphisms. This is the main result of this section. Recall that, when $N \nor G$ and the orders of $G/N$ and $N$ are coprime, every $G$-invariant $\theta \in \irr N$ admits a canonical extension $\widehat \theta \in \irr G$. The extension $\widehat \theta$ is canonical, in the sense that it is the only extension with the same determinantal order as $\theta$ (see \cite[Cor 6.2]{Nav18}, for instance).  

\begin{teo}\label{thm:OWrevisited} Suppose that $D$ acts coprimely on $G$. Write $A = GD$ to denote the semidirect product of $G$ and $D$. Let $p$ be a prime. Let $K \sbs G$ be such that $K \nor A$ and $p$ does not divide the order of $K$. Let $P \in \syl p{G}$ be $D$-invariant and such that $KP \nor A$. Let $Z$ be a $p$-subgroup of $G$ such that $Z \nor A$. Let $\lambda \in \irr{Z}$ be $A$-invariant. Write $H=\norm{G}{P}$, and $C=\cent{K}{P}$. Let $\theta \in \irr{K}$ be $P$-invariant and let $\theta^* \in \irr{C}$ be the Glauberman correspondent of $\theta$. Then
    \begin{equation*}
        |\irra {p',D}{G|\theta \times \lambda}|=|\irra{p',D}{H|\theta^* \times \lambda}|.
    \end{equation*}
\end{teo}
\begin{proof} 
Write $B=\norm{A}{P}$. 
By the Frattini argument, $A=(KP)B=KB$ and $G=(KP)H=KH$. Notice that $Z \sbs P$. 
Since $P$ is $D$-invariant, we have that $B = HD$. The following diagram depicts the relevant character and group-theoretic situation.
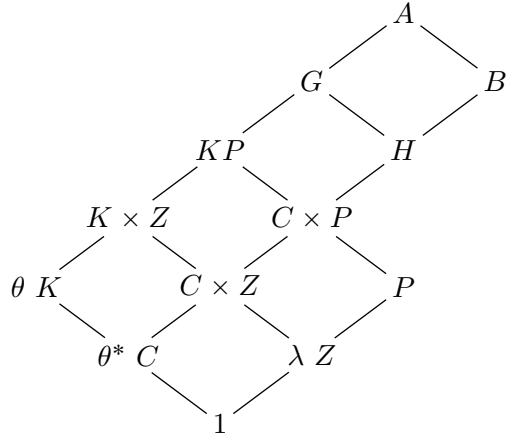
\begin{center}
    \begin{tikzpicture}
        \node (top) at (0,0) {$A$};

        \node [below left of=top, yshift= -0.2cm, xshift=-0.5cm] (top1)  {$G$};
        \node [below left of=top1, yshift= -0.2cm, xshift=-0.5cm] (top2)  {$KP$};
        \node [below left of=top2, yshift= -0.2cm, xshift=-0.5cm] (left1)  {$K\times Z$};
        \node [below left of=left1, yshift= -0.2cm, xshift=-0.5cm] (left2)  {$ \theta \  K$};

        \node [below right of=top, xshift=0.5cm, yshift= -0.2cm] (right15) {$B$};
        \node [below left of=right15, xshift=-0.5cm, yshift= -0.2cm] (right1) {$H$};
    
        \node [below left of=right1, yshift= -0.2cm, xshift=-0.5cm] (top3)  {$C\times P$};
        \node [below left of=top3, yshift= -0.2cm, xshift=-0.5cm] (right2)  {$C\times Z$};
        \node [below left of=right2, yshift= -0.2cm, xshift=-0.5cm] (right3)  {$ \theta^* \ C$};

        \node [below right of=top3, xshift=0.5cm, yshift= -0.2cm] (right14) {$P$};
        \node [below left of=right14, xshift=-0.5cm, yshift= -0.2cm] (right12) {$ \lambda \ Z$};
        \node [below left of=right12, yshift= -0.2cm, xshift=-0.5cm] (right13)  {$1$};

        \draw [black, line width=0.5pt, shorten <=-2pt, shorten >=-2pt] (top) -- (top1);
        \draw [black, line width=0.5pt, shorten <=-2pt, shorten >=-2pt] (top2) -- (top1);
        \draw [black, line width=0.5pt, shorten <=-2pt, shorten >=-2pt] (top3) -- (top2);
        \draw [black, line width=0.5pt, shorten <=-2pt, shorten >=-2pt] (top3) -- (right1);
        \draw [black, line width=0.5pt, shorten <=-2pt, shorten >=-2pt] (top) -- (right15);
        \draw [black, line width=0.5pt, shorten <=-2pt, shorten >=-2pt] (left1) -- (top2);
        \draw [black, line width=0.5pt, shorten <=-2pt, shorten >=-2pt] (right2) -- (top3);
    
        \draw [black, line width=0.5pt, shorten <=-2pt, shorten >=-2pt] (left1) -- (right2);
        \draw [black, line width=0.5pt, shorten <=-2pt, shorten >=-2pt] (left1) -- (left2);
        \draw [black, line width=0.5pt, shorten <=-2pt, shorten >=-2pt] (right1) -- (right15);
        \draw [black, line width=0.5pt, shorten <=-2pt, shorten >=-2pt] (right1) -- (top1);

        \draw [black, line width=0.5pt, shorten <=-2pt, shorten >=-2pt] (left2) -- (right3);
        \draw [black, line width=0.5pt, shorten <=-2pt, shorten >=-2pt] (right2) -- (right3);

        \draw [black, line width=0.5pt, shorten <=-2pt, shorten >=-2pt] (right12) -- (right2);
        \draw [black, line width=0.5pt, shorten <=-2pt, shorten >=-2pt] (right12) -- (right14);
        \draw [black, line width=0.5pt, shorten <=-2pt, shorten >=-2pt] (top3) -- (right14);
    
        \draw [black, line width=0.5pt, shorten <=-2pt, shorten >=-2pt] (right13) -- (right3);
        \draw [black, line width=0.5pt, shorten <=-2pt, shorten >=-2pt] (right13) -- (right12);
    \end{tikzpicture}
\end{center}
Notice that a character of $G$ is $D$-invariant if, and only if, it is $A$-invariant; and similarly a character of $H$ is $D$-invariant if, and only if, it is $B$-invariant.

We next show that we may assume that both $\theta$ and $\theta^*$ are $B$-invariant (and hence that $\theta$ is $A$-invariant).
  By \cite[Lem. 2.10]{Nav18}, the Glauberman correspondence is $B$-equivariant (because $B$ acts on $KP$ by automorphisms stabilizing $P$). Hence $B_{\theta}=B_{\theta^*}$, and $H_{\theta}=H_{\theta^*}$. Notice that $A_{\theta \times \lambda}=A_{\theta} \cap A_{\lambda}=A_{\theta}$, and $B_{\theta^* \times \lambda}=B_{\theta^*}$, hence $B_{\theta \times \lambda}=B_{\theta}=B_{\theta^*}=B_{\theta^* \times \lambda}$.  Since $A=GB$, $H = G  \cap B$, $KZ \nor A$ and $CZ \nor B$, we can apply the second part of Lemma \ref{lem:technical}. 
    Therefore it will suffice to show that
    \begin{equation*}
        |\irra{p',A_{\theta}}{G_{\theta}|\theta \times \lambda}|=|\irra{p',B_{\theta}}{H_{\theta}|\theta^* \times \lambda}|.
    \end{equation*}
Notice that $A_\theta=G_\theta B_\theta$, by the Frattini argument. 
Since $(|B_\theta:H_\theta|, |H_\theta|)=1$, by the Schur-Zassenhaus theorem we can choose
$E$ a complement of $H_\theta$ in $B_\theta$.
 Then $E$ acts coprimely on $G_\theta$,  $P$ is $E$-invariant because $E\sbs B$ and $A_\theta=G_\theta E$. It suffices to prove that
         \[ |\irra {p',E}{G_\theta|\theta \times \lambda}|=|\irra{p',E}{H_\theta|\theta^* \times \lambda}|, \]
         and our first claim is proven.

Let $\widehat \theta$ be the canonical extension of $\theta$ to $KP$ and write $\widehat{\theta^*}=\theta^* \times 1_P \in \irr{CP}$. It is clear that $\widehat{\theta^*}$ is $B$-invariant. Moreover, $\theta$ is $A$-invariant and hence so is $\widehat \theta$. 
Since the groups $KP/K$, $CP/C$, and $P$ are isomorphic, we can identify their characters by inflation and deflation. 
Let $\mathcal{A}$ be a complete set of representatives of the orbits of the action of $H$ on the linear characters of $KP/K$ which lie over $1_K \times \lambda \in \irr{KZ/K}$. Within the proof of \cite[Theorem 8.11]{Nav18} it is proven that:
    \begin{align*}
        \irra{p'}{G|\theta \times \lambda}&= \mathop{\dot{\bigcup}}_{\mu \in \mathcal{A}} \irr{G|\widehat \theta \mu}\, .\\
        \irra{p'}{H|\theta^* \times \lambda}&= \mathop{\dot{\bigcup}}_{\mu \in \mathcal{A}} \irr{H|\widehat{\theta^*} \mu}\, .
    \end{align*}
In particular, we have that:
    \begin{align*}
        \irra{p', A}{G|\theta \times \lambda}&= \mathop{\dot{\bigcup}}_{\mu \in \mathcal{A}} \irra A {G|\widehat \theta \mu}\, \\
        \irra{p', B}{H|\theta^* \times \lambda}&= \mathop{\dot{\bigcup}}_{\mu \in \mathcal{A}} \irra B{H|\widehat{\theta^*} \mu}\, .
    \end{align*}
  Thus, in order to prove the statement, it suffices to show that, for every $\mu \in \mathcal{A}$,
    \begin{equation}\label{eq:mu}
        |\irra{A}{G|\widehat \theta \mu}|=|\irra{B}{H|\widehat{\theta^*} \mu}|.
    \end{equation}
By using the first part of Lemma \ref{lem:technical}, and reasoning as before, we may assume $\mu$ to be $A$-invariant when proving Equation \ref{eq:mu} for $\mu\in \mathcal A$.
    In fact, recall that $\widehat \theta$ is $A$-invariant and $\widehat{\theta^*}$ is $B$-invariant, so
 by the Gallagher correspondence, $A_{\widehat \theta \mu}=A_{\mu}$ and $B_{\widehat \theta \mu}=B_{\widehat{\theta^*} \mu}$. 
Then it suffices to show that
\[  |\irra{A_\mu} {G_\mu |\widehat \theta \mu}| =|\irra{B_\mu} {H_\mu |\widehat{\theta^*}\mu}| \, . \]
As before, we can choose $E$ a complement of $H_\mu$ in $B_\mu$ so that $E$ acts coprimely on $G_\mu$,  $P$ is $E$-invariant and $A_\mu=G_\mu E$.

Next we want to show that it in order to prove Equation \ref{eq:mu} for $\mu\in \mathcal A$ it is enough to show that
 \begin{equation}\label{eq:hat}
     |\irra D {G|\widehat \theta}| = |\irra D {H|\widehat{\theta^*}}|\, .
         \end{equation}
  Notice that $(|G:KP|,|KP:K|)=1$, therefore $\mu$ admits a canonical extension $\widehat \mu \in \irr G$. We have that $\widehat \mu$ is $A$-invariant.
  (Under our hypotheses, $\mu$ actually extends to $A$.)
  Then \cite[Thm. 6.16]{Isa76} (or \cite[Lem. 5.8]{Nav18}) tells us that that the map
    \begin{align*}
        \irr{G|\widehat \theta} & \rightarrow \irr{G|\widehat \theta \mu} \\
        \beta & \mapsto \beta\widehat \mu
    \end{align*}
    is a bijection. 
    Since $\widehat \mu$ is $A$-invariant, this bijection is also $A$-equivariant.
   Hence
    \begin{equation*}
        |\irra{A}{G|\widehat \theta\mu}|=|\irra{A}{G|\widehat \theta}|.
    \end{equation*}
    Similarly, we obtain that
    \begin{equation*}
        |\irra{B}{H|\widehat{\theta^*}\mu}|=|\irra{B}{H|\widehat{\theta^*}}|.
    \end{equation*} 
Therefore, in order to prove Equation \ref{eq:mu} for every $\mu \in \mathcal A$, it will be enough to show that
\[ |\irra{A}{G|\widehat \theta}|=|\irra{B}{H|\widehat{\theta^*}}|\, . \]
Recall that $A = GD$ and $P$ is $D$-invariant, so we have that $B = HD$; hence, we want to prove that 
    \[
     |\irra D {G|\widehat \theta}| = |\irra D {H|\widehat{\theta^*}}|\, .\]
         
We work to prove Equation \ref{eq:hat}. We will show that the proof reduces to the situation described by Corollary \ref{cor:OW}. 
We begin by depicting the group theory that arises in this context.

Since $PK\nor G$, we have that $CP/C$ is a normal Sylow $p$-subgroup of $H/C$. By the Schur-Zassenhaus theorem, 
the group $H/C$ has a $p$-complement $U/C$, that is $U\leq H$  with $U(CP)=H$ and $U \cap CP=C$.
In particular, $U$ has order coprime to $p$. Since $K \nor G$ and $PU=UP \leq G$, we have that $(KP)(KU)\leq G$. 
The group $(KP)(KU)$ contains both $U(CP)=H$ and $K$, hence
we have that $(KP)(KU)=G$. Furthermore, 
we have that $KU \cap KP = K(U \cap P)=K$ and
$KU\cap H=U(K\cap H)=UC=U$, by Dedekind's lemma.
 Notice that by the Schur--Zassenhaus theorem, all complements of $CP/C$ in $H/C$ are $H$-conjugate and $D$ acts on them. By applying Glauberman's Lemma \cite[Lemma 13.8]{Isa76}, we can choose $U/C$ to be a $D$-invariant complement of $CP/C$ in $H/C$; in other words, $U \nor UD$ and $KU \nor KUD$.
The following figure depicts most of the information gathered above.


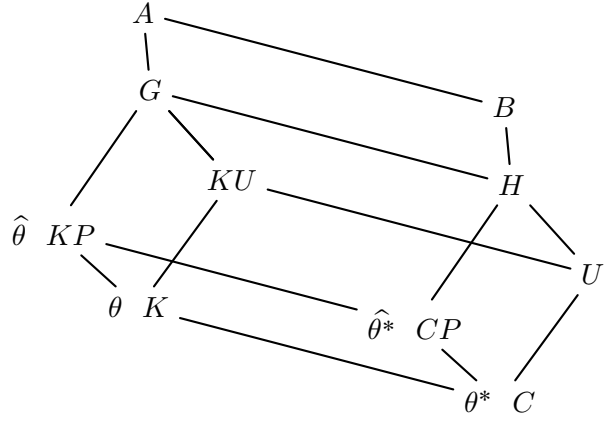
\begin{center}
\begin{tikzpicture}
[line cap=round, line join=round, thick, scale=1,
  x={(0.95cm,-0.25cm)}, y={(0.25cm,0.85cm)}, z={(0cm,1cm)}]
\node (A) at (-0.1,0,3) {$A$};
\node (B) at (4.9,0,3) {$B$};
\node (G) at (0,0,2) {$G$};
\node (H) at (5,0,2) {$H$};
\node (KP) at (-1,-1.3,1) {$\widehat \theta \ \ KP$};
\node (K) at (0.1,-1,0) {$\theta \  \ K$};
\node (C) at (5.1,-1,0) {$ \theta^* \  \ C$};
\node (U) at (6.1,0.1,1) {$U$};
\node (KU) at (1.1,0.1,1) {$KU$};
\node (CP) at (4,-1.3,1) {$\widehat{\theta^*} \ \ CP$};

\draw (A)--(B)--(H)--(G)--cycle;
\draw (KP)--(K)--(KU)--(G)--cycle;
\draw (H)--(U);
\draw (K)--(C)--(U);
\draw (C)--(CP);
\draw (KU)--(U);
\draw (KP)--(CP);
\draw (G)--(KU) (H)--(CP);
\draw (G) -- (KP);
\draw (G) -- (A);
\end{tikzpicture}
\end{center}
Recall that $\widehat\theta_K=\theta$ and $(\widehat{\theta^*})|_C=\theta^*$.
Moreover $G=(KP)(KU)$ with $KP\nor G$ and $KP\cap KU=K$, 
as well as that $H=(CP)U$ with $CP\nor H$ and $CP\cap U=C$. Also $A=G(KUD)$ and $G\cap KUD=KU$. Using 
\cite[Lem. 10.5]{Isa73} (that appears as \cite[Lem. 6.8]{Nav18}), we have that character restriction yields a bijection
\[ \alpha \colon \irr {G | \widehat \theta} \to \irr {KU | \theta}\, .\]
Since $KU\nor KUD$ and both $\widehat{\theta}$ and $\theta$ are $D$-invariant, we have 
that $\alpha$ is $D$-equivariant. Similarly, character restriction
\[ \beta \colon \irr{H|\widehat{\theta^*}} \to \irr{U |\theta^*}\, \]
also yields a $D$-equivariant bijection. Hence $\displaystyle |\irra{D}{G|\widehat \theta}|=|\irra{D}{KU|\theta}|$  and 
  $\displaystyle |\irra{D}{H|\widehat{\theta^*}}|=|\irra{D}{U|\theta^*}|$. 
By applying Corollary \ref{cor:OW} we have that 
\[ |\irra{D}{KU|\theta}|=|\irra{D}{U|\theta^*}|\, . \]
This proves Equation \ref{eq:hat} and therefore concludes our proof.
\end{proof}

\section{The relative McKay conjecture with coprime group automorphisms}\label{reduction}

The aim of this section is to prove Theorem B from the Introduction. We will actually prove a projective version of Theorem B 
by carefully following the relevant parts of the proof of \cite[10.26]{Nav18}. We will use the notion of strong character triple isomorphism, see \cite[Problem 11.13]{Isa76} or \cite[Problem 5.4]{Nav18}.

\begin{teo}\label{thm:projective-thmB}
Let $A$ act coprimely on $G$. Let $p$ be a prime and $P \in \syl p G$ be an $A$-invariant Sylow $p$-subgroup of $G$. Let $N$ be a normal $A$-invariant subgroup of $G$. Suppose that $G/N$ is $p$-solvable. If $\mu \in \irra {p'} N$ is $P$-invariant, then 
\[ |\irra {p', A} {G|\mu} | = |\irra {p', A} {\norm G PN |\mu}|\, . \]
\end{teo}

\begin{proof}
Write $GA$ to denote the semidirect product of $G$ and $A$. Since $\norm {GA} P \supseteq A$, we have that $\norm {GA}P = \norm G P A$. Write $H = \norm G P N$, so that $GA = G(HA)$ and $G \cap HA = H$.
We argue by induction on $|GA:N|$. 

\smallskip

We first show that we can assume $\mu$ to be $GA$-invariant. 
If both sets $\irra {p', A} {G|\mu}$ and $\irra {p', A} {\norm G PN |\mu}$ are empty there is nothing to prove. 
Assume that $(GA)_\mu<GA$.

Suppose that there is some $\chi \in \irra{p', A}{G|\mu}$. Arguing as in the proof of Lemma \ref{lem:technical}, $GA=G(GA)_\mu$
and
\[ |\irra {p', A}{G|\mu}|=|\irra{p', GA}{G|\mu}|=|\irra{p', (GA)_\mu}{G_\mu|\mu}|\, .\]
We now show that $HA=H(HA)_\mu$.
Let $b \in HA$. Since $HA/N=\norm{GA/N}{PN/N}$,  for every $x \in P$ we have that there are $y\in P$ and $n \in N$ such that $bxb^{-1}=ny$.
In particular, 
$\mu^{b}$ and $\mu$ are $P$-invariant characters and lie under $\chi$. 
By \cite[Lemma 9.3]{Nav18}, there is some $h \in \norm {G} P$ such that $\mu^{bh}=\mu$,
so $bh \in (HA)_\mu$ hence $HA=H(HA)_\mu$. As before, this implies that
 \[ |\irra {p', A}{H|\mu}|=|\irra{p', HA}{H|\mu}|=|\irra{p', (HA)_\mu}{H_\mu|\mu}|\, .\]
Notice also that $GA=G(HA)=GH(HA)_\mu=G(HA)_\mu$ and $G\cap(HA)_\mu=H_\mu$.
By the Schur-Zassenhaus theorem we can choose  $D$ a complement of $\norm G P_\mu$ in $\norm{GA} P_\mu$.
Since $H_\mu=N\norm GP_\mu$ and $(HA)_\mu=N\norm {GA}P_\mu$, then $D$ acts coprimely on $G_\mu$ stabilizing $P$
and 
$(GA)_\mu=G_\mu D$.
By induction
\[ |\irra{p', (GA)_\mu}{G_\mu|\mu}|=|\irra{p', D}{G_\mu|\mu}|=|\irra{p', D}{H_\mu|\mu}|=|\irra{p', (HA)_\mu}{H_\mu|\mu}|\, . \] 
Suppose now that there is some $\psi \in \irra{p', A}{H|\mu}$. Then $HA=H(HA)_\mu$ and consequently $GA=G(GA)_\mu$.
Arguing as in  Lemma \ref{lem:technical},
\begin{align*}
|\irra {p', A}{G|\mu}|&=|\irra{p', GA}{G|\mu}|=|\irra{p', (GA)_\mu}{G_\mu|\mu}|, \\
|\irra {p', A}{H|\mu}|&=|\irra{p', HA}{H|\mu}|=|\irra{p', (HA)_\mu}{H_\mu|\mu}|.
\end{align*}
As before, we can choose $D$ a complement of $\norm G P_\mu$ in $\norm {GA} P_\mu$, so that 
 $D$ acts coprimely on $G_\mu$ stabilizing $P$ and $G_\mu D=(GA)_\mu<GA$.
By induction 
\[ |\irra{p', (GA)_\mu}{G_\mu|\mu}|=|\irra{p', D}{G_\mu|\mu}|=|\irra{p', D}{H_\mu|\mu}|=|\irra{p', (HA)_\mu}{H_\mu|\mu}|\, ,\]
and we are done in this case too.

\smallskip

Next, we claim that we can assume that $N \subseteq \zent{GA}$.
By \cite[Problem 5.4.(b)]{Nav18} we have that the character triple $(GA,N,\mu)$ is strongly isomorphic to a character triple $((GA)^*,N^*,\phi)$, where $N^*$ is a central subgroup of $(GA)^*$ and $\phi \in \irr{N^*}$ (we denote a strong isomorphism $(^*, \sigma)$ as in \cite[Def. 11.23]{Isa76}). 
  If $N \sbs K \leq GA$ we will denote by $K^*$ the subgroup of $(GA)^*$ containing $N^*$ such that $(K/N)^*=K^*/N^*$ under the group isomorphism $^*\colon GA/N \to (GA)^*/N^*$. Notice that $N^* \sbs G^*$. Since $N^*$ is central in $(GA)^*$, there exist a unique $P^* \in \syl p{G^*}$ such that $(PN)^*/N^*=P^*N^*/N^*$.
 Using that $\norm {G/N}{PN/N}=H  /N$ and $\norm {G^*/N^*}{P^*N^*/N^*}=\norm {G^*}{P^*}N^*/N^*=\norm {G^*}{P^*}/N^*$, we conclude that 
 $$\norm {G^*}{P^*}=H^*\, .$$
 Similarly, we can also prove that 
 $$\norm {(GA)^*}{P^*}=(HA)^*\, .$$
 Write $\pi$ to denote the set of prime divisors of $|G|$, and let $\pi'$ be the complement of $\pi$ in the set of prime numbers. 
 Since $N^* \subseteq \zent{(GA)^*}$, we can write $N^* = M \times Z$ as a direct product, where $M$ is a Hall $\pi$-subgroup of $N^*$ and $Z$ is a Hall $\pi'$-subgroups of $N^*$.
 By the Schur-Zassenhaus theorem, $M$ has a complement $A^*$ in $(NA)^*$. (In fact, $(NA)^*$ is the direct product of $M$ and $A^*$.) 
 As $GA=G(NA)$, we get that $(GA)^* = G^*A^*$ and also $(HA)^* = \norm {(GA)^*} {P^*} = H^*A^*$. In particular $P^*$ is $A^*$-invariant. 
 
The strong character triple isomorphism preserves character degree ratios and therefore induces bijections
 $$\sigma_G \colon \irra {p'} {G|\mu} \to \irra {p'} {G^*|\phi} \text{ \  \ and  \  \ } \sigma_{H}\colon \irra {p'} {H|\mu} \to \irra {p'} {\norm {G^*}{P^*} | \phi}\, .$$
 Notice that $\sigma_G$ and $\sigma_H$ are equivariant with respect to the action of $AN/N\cong A^*N^*/N^*$ 
 on characters.
 Therefore 
 \[ |\irra {p', A}{G|\mu}|=|\irra {p', A^*}{G^*|\phi}| \text{ \ \ and \  \ } |\irra {p', A}{H|\mu}|=|\irra {p', A^*}{\norm {G^*}{P^*}|\phi}|\, .\]
 The orders of $A^*$ and $G^*$ are not necessarily coprime. We need to slightly modify the character triple isomorphism using parts (c) and (d) of \cite[Lemma 5.8]{Nav18}.

 Write $\varphi_\pi = \varphi|_M$ and $\varphi_{\pi'} = \varphi|_Z$ so that $\varphi = \varphi_\pi \times \varphi_{\pi'}$. When necessary, we will view $\varphi_\pi, \varphi_{\pi'}\in \irr{N^*}$ and $\varphi=\varphi_\pi\varphi_{\pi'}$.  The character $\varphi_{\pi'}$ admits a canonical extension $\lambda$ to $G^*$ by \cite[Corollary 6.3]{Nav18}. In particular, $\lambda$ is $A^*$-invariant.  The following diagram might be useful.
 \begin{center}
 \begin{tikzpicture}

    \node (top) at (0,0) {$(GA)^*$};

    \node [below right of=top, yshift= -0.05cm, xshift=0.2cm] (top1_phantom) {}; 
    \node [below right of=top1_phantom, yshift= -0.04cm, xshift=0.2cm] (top2)  {$(NA)^*$};

    \node [below right of=top2, yshift= -0.15cm, xshift=0.4cm] (left1_phantom)  {};
    \node [below right of=left1_phantom, yshift= -0.10cm, xshift=0.4cm] (left2)  {$ A^*$};

    \node [below left of=top, xshift=-0.4cm, yshift= -0.2cm] (right15) {$\lambda$ $G^*$};
    
    \node [below right of=right15, xshift=0.2cm, yshift= -0.05cm] (right1_phantom) {}; 
    \node [below right of=right1_phantom, yshift= -0.04cm, xshift=0.2cm] (top3)  {$\varphi$ $N^*$};

    \node [below right of=top3, yshift= -0.15cm, xshift=0.4cm] (right2_phantom)  {};
    \node [below right of=right2_phantom, yshift= -0.10cm, xshift=0.4cm] (right3)  {$\lambda|_Z=\varphi_{\pi'}$ $ Z$};

    \node [below left of=top3, xshift=-0.4cm, yshift= -0.2cm] (right14) {$\varphi_\pi$ $M$};

    \node [below right of=right14, xshift=0.4cm, yshift= -0.15cm] (right12_phantom) {};
    \node [below right of=right12_phantom, yshift= -0.10cm, xshift=0.4cm] (right13)  {$1$};

    \draw [black, line width=0.5pt, shorten <=-2pt, shorten >=-2pt] (top) -- (top2); 
    \draw [black, line width=0.5pt, shorten <=-2pt, shorten >=-2pt] (top3) -- (top2); 
    \draw [black, line width=0.5pt, shorten <=-2pt, shorten >=-2pt] (top2) -- (left2); 
    
    \draw [black, line width=0.5pt, shorten <=-2pt, shorten >=-2pt] (top) -- (right15); 
    \draw [black, line width=0.5pt, shorten <=-2pt, shorten >=-2pt] (right15) -- (top3); 
    \draw [black, line width=0.5pt, shorten <=-2pt, shorten >=-2pt] (top3) -- (right3); 
    
    \draw [black, line width=0.5pt, shorten <=-2pt, shorten >=-2pt] (left2) -- (right3); 
    \draw [black, line width=0.5pt, shorten <=-2pt, shorten >=-2pt] (top3) -- (right14); 
    \draw [black, line width=0.5pt, shorten <=-2pt, shorten >=-2pt] (right14) -- (right13);
    \draw [black, line width=0.5pt, shorten <=-2pt, shorten >=-2pt] (right13) -- (right3); 

\end{tikzpicture}
 \end{center}
Notice that $\varphi(\lambda^{-1})_{N^*}=\varphi_\pi$.
 By \cite[Lemma 5.8.(d)]{Nav18}, we can replace $(G^*, N^*, \varphi)$ with $(G^*, N^*, \varphi_{\pi})$ by setting the map $\irr{U|\varphi} \to \irr{U|\varphi|_{\pi}}$ for each $N^* \leq U \leq G^*$ as $\beta \mapsto \beta(\lambda|_{U})^{-1}$. These character bijections are $A^*$-equivariant because $\lambda$ is $A^*$-invariant.
 By using \cite[Lemma 5.8.(d)]{Nav18}, we can replace $(G^*, N^*, \varphi_{\pi} )$ with $(G^*/Z, N^*/Z, \mu^*)$, where $\mu^*(Zn)=\varphi(n)$ for every $n \in N^*$, by viewing every character of $N^*\leq U \leq G^*$ above $\varphi_{\pi}$ as a character of the quotient $U/Z$. These character bijections are $A^*$-invariant. 
 
In the group $G^*/Z$ we have that $\norm {G^*}{P^*Z/Z} = \norm {G^*}{P^*}/Z$. Also $N^*/Z \subseteq \zent{G^*A^*/Z}$, 
and $A^*/Z$ acts coprimely on $G^*/Z$. 
Moreover $|GA:N| = |G^*A^*/Z: N^*/Z|$,
     \begin{align*}
        & |\irra {p', A}{G|\mu}| = |\irra {p', A^*/Z}{G^*/Z| \mu^*}| \text{, and} \\
        & |\irra {p', A}{H|\mu}| = |\irra {p', A^*/Z}{\norm{G^*}{P^*}/Z| \mu^*}|.
    \end{align*}
The claim is then proven.
 
\smallskip 

 Since $N \sbs \zent {GA}$, we have that $H=\norm{G}{P}$ and $HA=\norm {GA} P$.
 Let $L/N$ be a minimal normal subgroup of $GA/N$ contained in $G/N$. 
 Next, we show by induction that 
 \[ |\irra {p', A}{G|\mu}|=|\irra {p', A}{LH| \mu}|\, .\] 
 Notice that $H$ acts on set $\irra{p', P}{L|\mu}$ of $P$-invariant irreducible characters of $L$ of degree coprime to $p$ lying over $\mu$.
 Let $\mathcal{A}$ be a complete set of representatives of the orbits of the action of $H$ on $\irra{p', P}{L|\mu}$. Then, by using \cite[9.3]{Nav18}, we have that
    \begin{align*}
        \irra{p',A}{G|\mu}&=\mathop{\dot{\bigcup}}_{\theta \in \mathcal{A}} \irra{p',A}{G|\theta}, \\
        \irra{p',A}{LH|\mu}&=\mathop{\dot{\bigcup}}_{\theta \in \mathcal{A}} \irra{p',A}{LH|\theta},
 \end{align*}
 where unions above are disjoint. Since $|GA:L|<|GA:N|$, and $L \nor GA$, by induction we have that
    \begin{equation*}
        |\irra{p',A}{G|\theta}|=|\irra{p',A}{LH|\theta}|,
    \end{equation*}
    for every $\theta \in \mathcal{A}$. Therefore $|\irra {p', A}{G|\mu}|=|\irra {p', A}{HL| \mu}|$, as claimed. 

Suppose now that $LH<G$, then $L(HA)< GA$. By
induction, since $LH \nor L(HA)$ and $|LHA:N|<|GA:N|$, we would get that
\[ |\irra {p', A}{LH|\mu}|=|\irra{p', A}{H|\mu}|\, .\]
Hence, we may assume that $LH=G$. In particular $PL\nor G$. Also $GA=L(HA)$, and as $P$ is $A$-invariant, we have that actually $PL\nor GA$.
By hypothesis, $G/N$ is $p$-solvable, then $L/N$ is either a $p$-group or a $p'$-group.
Suppose that $L/N$ is a $p$-group, then $L\sbs PN$. Hence $G=LH=PNH=H$ and there is nothing to prove. 

Therefore, we can assume that $L/N$ is a $p'$-group.
Write $Z=P\cap N$ and $\lambda=\mu_Z$.  Notice that $\lambda$ is $A$-invariant. Since $Z$ is a central Sylow $p$-subgroup of $L$, we can write $L=K \times Z$, with $K$ a $p'$-subgroup. Notice $Z=P\cap L$. Observe that $K$ is characteristic in $L$, hence $K \nor GA$. Moreover $KP=LP\nor GA$. 

We work to conclude this proof by applying Theorem \ref{thm:OWrevisited}. 
Write $M=N \cap K$ and $\xi=\mu_M$. We have that $N=M \times Z$. 
Let $\mathcal{B}$ be a complete set of representatives of the orbits of the action of $H$ on $\irra P {K|\xi}$ and let $^*:\irra{P}{K} \rightarrow \irr{\cent K P}$ be the Glauberman correspondence with respect to the action of $P$ on $K$. 
We have that $^*$ maps $\irra{P}{K|\xi}$ onto $\irr{\cent K P|\xi}$ because for every $\theta \in \irra{P}{K|\xi}$, its Glauberman correspondent $\theta^*$ is an irreducible constituent of the restriction $\theta_{\cent K P}$, by \cite[Thm. 2.9]{Nav18}.
Since the Glauberman correspondent is equivariant with respect to the action of $H$ on characters \cite[Lem. 2.10]{Nav18},
we have that $\mathcal{B^*}=\{\theta^* \ | \  \theta \in \mathcal{B}\}$ is a complete set of representatives of the orbits of the action of $H$ on $\irr{\cent{K}{ P}|\xi}$. Another application of \cite[9.3]{Nav18}, noticing that $\mu=\xi \times \lambda$, yields
    \begin{align*}
        & \irra{p',A}{G|\mu}=\mathop{\dot{\bigcup}}_{\theta \in \mathcal{B}} \irra{p',A}{G|\theta \times \lambda}, \\
        & \irra{p',A}{H|\mu}=\mathop{\dot{\bigcup}}_{\theta \in \mathcal{B}} \irra{p',A}{H|\theta^* \times \lambda}.
    \end{align*}
   Since $KP=LP\nor GA$, we can apply Theorem $\ref{thm:OWrevisited}$ to conclude that 
\[ |\irra{p',A}{G|\theta \times \lambda}|=|\irra{p',A}{H|\theta^* \times \lambda}|\, ,\]
for every $\theta \in \mathcal{B}$. This finishes the proof.
\end{proof}

Note that Theorem B follows by letting $N=1$ in Theorem \ref{thm:projective-thmB}.

\end{document}